\documentclass{amsart}
\usepackage[mathscr]{euscript}
\newtheorem{theorem}{Theorem}[section]
\newtheorem{lemma}[theorem]{Lemma}

\theoremstyle{definition}
\newtheorem{definition}[theorem]{Definition}

\newtheorem{corollary}{Corollary}[theorem]
\newtheorem{proposition}{Proposition}[theorem]
\theoremstyle{remark}

\numberwithin{equation}{section}

\begin{document}

\title{Descent of Properties of Rings and Pairs of Rings to Fixed Rings}


\author{Ravinder Singh}
\address{Department of Mathematics, Dr. B. R. Ambedkar National Institute of Technology Jalandhar}
\email{}
\thanks{singhr@nitj.ac.in}



\date{}

\dedicatory{}

\begin{abstract}
Let $G$ be a group acting via ring automorphisms on an integral domain $R.$ A ring-theoretic property of $R$ is said to be $G$-invariant, if $R^G$ also has the property, where $R^G=\{r\in R \ | \ \sigma(r)=r \ \text{for all} \ \sigma\in G\},$ the fixed ring of the action. In this paper we prove the following classes of rings are invariant under the operation $R\rightarrow R^G:$ locally pqr domains, Strong G-domains, G-domains, Hilbert rings, $S$-strong rings and root-closed domains. Further let $\mathscr{P}$ be a ring theoretic property and $R\subseteq S$ be a ring extension. A pair of rings $(R,S)$ is said to be a $\mathscr{P}$-pair, if $T$ satisfies $\mathscr{P}$ for each intermediate ring $R\subseteq T\subseteq S.$ We also prove that the property $\mathscr{P}$ descends from $(R,S)\rightarrow (R^G, S^G)$
in several cases. For instance, if $\mathscr{P}=$ Going-down, Pseudo-valuation domain and ``finite length
of intermediate chains of domains", we show each of these properties successfully transfer from $(R,S)\rightarrow (R^G, S^G).$
\end{abstract}

\maketitle
\section{Introduction}\label{intro}
All the rings considered in this note are commutative, are assumed to contain unity and all ring homomorphisms are unital. We denote by $\text{Spec}(R)$ (respectively, $\text{Max}(R)$) the set of prime ideals (respectively, the set of maximal ideals) of $R,$ $\text{Aut}(R)$ the group of
automorphisms of a ring $R$ and $U(R)$ the group of units of $R.$ By an overring of a ring $R,$ we mean a subring of the total quotient ring of $R$ containing $R.$ If $R$ is an integral domain, then the total quotient ring of $R$ is same as the quotient field of $R,$ which we denote by $\text{qf}(R).$ If $R\subseteq S$ is a ring extensions, then the set of all the $R$-subalgebras of $S$ (i.e. of rings $T$ such that $R\subseteq T\subseteq S$) is denoted by $[R,S].$ As in \cite[Page 28]{kap}, LO GU, GD and INC refer to the Lying-over, Going-up, Going-down and Incomparable properties of ring extensions respectively.  

All ``actions" of a group on a ring are assumed to be via ring automorphisms. Given a subgroup $G$ of $\text{Aut}(R),$ the fixed ring of this action or the ring of $G$-invariants is denoted by $R^G=\{r\in R \ | \ \sigma(r)=r \ \text{for all} \ \sigma\in G\}.$ A ring-theoretic property of $R$ is said to be $G$-invariant, if $R^G$ also has the property. For a ring extension $R\subseteq S$ and $G$ a subgroup of $\text{Aut}(S),$ we assume that $\sigma(R)\subseteq R$ for all $\sigma\in G.$ It then follows that $R^G=R\cap S^G.$ The orbit of $s\in S$ under the action of $G$ is denoted by $\mathscr{O}_s=\{\sigma(s)\ | \ \sigma\in G\}.$ Following Glaz \cite{glaz}, we say that the action of $G$ is locally finite on $S$ if $\mathscr{O}_s$ is finite for all $s\in S$ and finite, if $|G|$ is finite. Every finite action is locally finite but not conversely \cite{ds07a}. For a locally finite action of $G$  and $s\in S,$ we put $$\hat{s}=\sum_{\sigma\in G} \sigma(s),\quad \widetilde{s}=\prod_{\sigma\in G} \sigma(s).$$

For long time researchers have been interested in determining the $G$-invariant ring-theoretic properties for the action of a group $G$ on a ring $R.$ The earliest studies in this area were motivated by Nagata's solution of Hilbert's fourteenth problem \cite{ds06}. The authors of these studies were primarily interested in descent of ``finiteness" property of ring a $R$ to the fixed ring $R^G.$ For instance, Nagarajan \cite{nag} studied the action of group on Noetherian rings. Bergman \cite{berg} proved that the fixed ring of any finite group of automorphisms of a Dedekind domain is a Dedekind. More generally, he proved that for a finite group $G$ with $|G|\in U(R),$ the Noetherian property of a ring $R$ is $G$-invariant. Glaz \cite{glaz} studied the conditions under which the finiteness property of coherence and related homological properties of $R$ descend to $R^G.$ For more on results similar in ``spirit" to the aforementioned studies, the reader is referred to \cite{glaz} and references therein.

Most of the results mentioned above are homological in nature, rathar than ideal-theoretic. The ideal-theoretic study, in this area, was apparently started with the paper of Dobbs and Shapiro \cite{ds06} and continued in their subsequent papers \cite{ds07a,ds07b}. Several ring-theoretic properties are shown to preserve in the passage from a ring $R$ to $R^G.$ For instance, it is proved that if the action of $G$ is locally finite, then $R^G\subset R$ is an integral extension (see Lemma \ref{l1} below). Among other things, they also proved the $G$-invariance of integrally closed domains, Krull domains, Pr\"{u}fer domains, going-down domains and divided domains. Motivated by these resuts Schmidt \cite{as17} studied the $G$-inavariance of ``minimal" ring extension $R\subset S$ under various assumptions. Recently Zeidi \cite{nz20} conducted a study of ring-theoretic properties of pairs of rings which are invariant under group actions.

The purpose of this note is to contribute to the circle of ideas initiated by Dobbs and Shapiro \cite{ds06}. In the Section \ref{s2}, we prove the invariance of several ring-theoretic properties under the operation $R\rightarrow R^G.$ For a locally finite group action, in Theorem \ref{2t1} we prove the $G$-invariance of locally pqr domain, a notion introduced by Ramaswamy and Visvanathan \cite{rv76}. This result is then used to deduce the invariance of  Strong G-domains under the operation $R\rightarrow R^G$ (see Corollary \ref{c1} below). The structre of non-Noetherian G-domains (and thus Strong G-domains as well) is quite mysterious. To quote Kaplansky \cite[Page 13]{kap}, `` For non-Noetherian domains the facts are more complex, and we seem to lack even a reasonable conjecture concerning the structure of general G-domains." We also provide a proof of $G$-invariance of G-domains that doesn't depend on the notion of locally pqr domain but is based on a result in \cite{kap} (see Theorem \ref{2t2}). In Theorem \ref{2t3}, we prove that if $R$ is a Hilbert ring then so is $R^G.$ The $G$-invariance of $S$-strong and root-closed domains is also obtained in Section \ref{s2}.

Section \ref{s3} is devoted to the investigation of ring-theoretic properties of pairs of rings that are invariant under the operation $R\rightarrow R^G$ in the sense of Zeidi \cite{nz20}. In Theorem \ref{3t1} we prove that the GD property (see Definition \ref{3d1} below) is preserved in the passage from a pair $(R,S)$ to $(R^G, S^G)$ under finite group actions. This result generalizes a result of Dobbs and Shapiro (\cite[Corollary 2.12]{ds06}), where they proved GD property for the pair $(R,\text{qf}(R)).$ Next we prove $G$-inavariance of Pseudo-valuation domain pair for finite group actions using a result of Jarboui and Trabelsi \cite{jt16}. For finite group actions, this provides a generalization of \cite[Theorem 2.12]{nz20}. Finally, the $G$-invariance of ``finite length of intermediate chains of domains" property is proved in Theorem \ref{3t3}.

\section{Properties of Rings}\label{s2}
In this section we shall prove several ring theoretic properties are stable under the operation $R\rightarrow R^G.$ We start by recalling a result which is fundamental for most of the proofs in this note. When $G$ is finite this is a well known result, see for instance \cite[ Exercise 12, Chapter 5]{am}. More generally, for a locally finite action of $G$ on $R,$ the following result was proved by Dobbs and Shapiro \cite{ds06}. 
\begin{lemma}[\cite{ds06}, Lemma 2.2]\label{l1}
If $G$ is locally finite, then $R$ is integral over $R^G.$
\end{lemma}

We now recall the notion of G-domain from \cite{kap}.
\begin{definition}
Let $R$ be an integral domain with the quotient field $K.$ We say $R$ is a \emph{G-domain} if it satisfies any of the following two equivalent statements:
(i) $K$ is a finitely generated ring over $R$ (ii) as a ring, $K$ can be generated over $R$ by single element.
\end{definition}  
In order to gain insight into the structure of G-domains and properties of overrings of G-domains, Ramaswamy and Visvanathan introduced a notion of ``Strong G-domain "\cite{rv76}. The class of Strong G-domains forms a subclass of the class of all Pr\"{u}fer domains. 
\begin{definition}[Ramaswamy and Visvanathan \cite{rv76}]
An integral domain $R$ is said to be \emph{Strong G-domain} if every overring of $R$ is of the form $R[1/t]$ for some nonzero $t\in R.$
\end{definition}
The study of Strong G-domain is closely related to a weaker notion of ``principal quotient ring", we recall
\begin{definition}[Ramaswamy and Visvanathan \cite{rv76}]
A quotient ring $R'$ of a domain $R$ with respect to a multiplicative set is called a \emph{principal quotient ring} (\emph{pqr}) of $R$ if $R'=R_M$ for a multiplicative set $M=\{1,t,t^2,\cdots, t^n,\cdots\}$ generated by a single element $t\in R.$ A domain $R$ is said
to be \emph{locally pqr} if for every prime ideal $\mathfrak{p}$ of $R$, the localization $R_{\mathfrak{p}}$ is a principal quotient ring of $R.$
\end{definition}
Trivially every strong G-domain is locally pqr and every locally pqr domain is a G-domain but not conversely \cite{rv76}. We now prove the invariance of locally pqr under the operation $R\rightarrow R^G.$
\begin{theorem}\label{2t1} 
Assume that $R$ is an integral domain and $G$ is locally finite. If $R$ is locally pqr domain then so is $R^G.$
\end{theorem} 
\begin{proof}
Let $R$ be an integral domain and $P\in\text{Spec}(R).$ Suppose $R$ is locally pqr domain. By \cite[Proposition 2.1 (3)]{rv76}, $R_P=R_t$ if and only if $t$ belongs to every prime ideal $Q$ of $R$ not contained in $P,$ for $t\in R.$ Therefore it is sufficient to prove the existence of an element in $R^G$ satisfying the preceding characterization of locally pqr domains. Fix $\mathfrak{p}\in\text{Spec}(R^G)$ and put $\widetilde{t}=\prod_{\sigma\in G}\sigma(t).$ Since $G$ is locally finite and permutes the elements of the orbit $\mathscr{O}_t,$ therefore $\widetilde{t}\in R^{G}.$ Let $\mathfrak{q}\in\text{Spec}(R^G)$ be an arbitrary prime such that $\mathfrak{q}$ is not contained in $\mathfrak{p}.$ Since $R^{G}\subset R$ is an integral extension by Lemma \ref{l1}, therefore by LO property for integral extensions \cite[Theorem 44]{kap}, there exists a prime ideal $Q$ such that $Q\cap R^G=\mathfrak{q}.$ Since $R$ is locally pqr domain, it follows that $\widetilde{t}\in Q\cap R^G=\mathfrak{q}.$ Thus $R^G$ is locally pqr by \cite[Proposition 2.1 (3)]{rv76}.
\end{proof}
\begin{corollary} \label{c1}
Assume that $R$ is an integral domain and $G$ is locally finite. If $R$ is Strong G-domain, then $R^G$ is also Strong G-domain.
\end{corollary}
\begin{proof}
By \cite[Theorem 3.5]{rv76}, $R$ is a Strong G-domain if and only if $R$ is a Pr\"{u}fer domain which is loaclly pqr. Suppose $R$ is strong G-domain, therefore it is both Pr\"{u}fer domain and loaclly pqr. Therefore we want to prove that $R^G$ is Pr\"{u}fer domain which loaclly pqr. But $R^G$ is locally pqr domain by Theorem \ref{2t1}. The fact that $R^G$ is also Pr\"{u}fer domain follows from \cite[Proposition 2.6 (a)]{ds06}. Thus $R^G$ is a Strong G-domain.
\end{proof}
If one is willing to relax the hypothesis (and hence the conclusion) of locally pqr domain and work with only G-domain property, then a proof of the $G$-invariance of G-domains can be provided based on \cite[Theorem 24]{kap}. For the sake of completeness we prove this result as well.
\begin{theorem}\label{2t2}
Assume that $R$ is an integral domain and $G$ is locally finite. If $R$ is a G-domain, then $R^G$ is also G-domain.
\end{theorem}
\begin{proof}
The \cite[Theorem 24]{kap} says that an integral domain $R$ is a G-domain if and only if there exists in the polynomial ring $R[x]$ an ideal $M$ which is maximal and satisfies $M\cap R=0.$ Suppose $R$ is a G-domain. By Theorem 24 in \cite{kap}, there exists a maximal ideal $M$ in $R[x]$ such that $M\cap R=0.$ Since $R^G\subset R$ an integral extension, it follows, by \cite[Exercise 9, Chapter 5]{am}, that $R^G[x]\subset R[x]$ is also an integral extension. Since $R^G[x]\subset R[x]$ integral extension, therefore $\mathfrak{m}=M\cap R^G[x]$ is also a maximal ideal. Now $$\mathfrak{m}\cap R^G=(M\cap R^G[x])\cap R^G=\mathfrak{m}\cap R^G=\mathfrak{m}\cap(R\cap R^G)=0.$$ Thus by \cite[Theorem 24]{kap}, $R^G$ is a G-domain.
\end{proof}

Closely related to the class of G-domains, but in some sense ``better behaved", is the class of Hilbert Rings (Bourbaki's terminology is Jacobson Rings). Recall that a prime ideal $\mathfrak{p}$ in a commutative ring $R$ is called \emph{G-ideal} if $R/\mathfrak{p}$ is a G-domain \cite{kap}. Obviously any maximal ideal is a G-ideal.
\begin{definition}
A commutative ring $R$ is called a \emph{Hilbert ring} (or \emph{Jacobson ring}) if every G-ideal in $R$ is maximal.
\end{definition}
\begin{theorem}\label{2t3}
Assume $G$ is locally finite. If $R$ is a Hilbert ring, then $R^G$ is also Hilbert ring.
\end{theorem}
\begin{proof}
We need to prove that every G-ideal in $R^G$ is maximal. To this end we first recall \cite[Theorem 27]{kap} which says that an ideal $\mathfrak{p}$ in a ring $R$ is a G-ideal if and only if it is the contraction of a maximal ideal in the polynomial ring $R[x].$ Let $\mathfrak{p}\in\text{Spec}(R^G)$ be a G-ideal. By the aforementioned result there exists a maximal ideal $\mathfrak{m}'$ in $R^G[x]$ such that $\mathfrak{p}=\mathfrak{m}'\cap R^G.$ Since $R^G[x]\subset R[x]$ is an integral extension by Lemma \ref{l1} and \cite[Exercise 9, Chapter 5]{am}, there exists a maximal ideal $\mathfrak{m}$ in $R[x]$ such that $\mathfrak{m}'=\mathfrak{m}\cap R^G[x].$ Now $$(\mathfrak{m}\cap R)\cap R^G=\mathfrak{m}\cap R^G=\mathfrak{m}\cap(R^G[x]\cap R^G)=(\mathfrak{m}\cap R^G[x])\cap R^G=\mathfrak{m}'\cap R^G=\mathfrak{p}.$$
That is, the prime ideal $\mathfrak{m}\cap R$ lies over $\mathfrak{p}.$ By \cite[Theorem 27]{kap}, $\mathfrak{m}\cap R$ is a G-ideal in $R$. Since $R$ is a Hilbert ring, therefore $\mathfrak{m}\cap R$ is also a maximal ideal. But $R^G\subset R$ is an integral extension, it follows that $(\mathfrak{m}\cap R)\cap R^G=\mathfrak{p}$ is also maximal. Thus $R^G$ is a Hilbert ring.
\end{proof} 
\begin{definition}
Let $R$ be a subring of a domain $S.$ A prime ideal $\mathfrak{p}$ in $R$ is called \emph{S-strong} if $x,y\in S,$ $xy\in\mathfrak{p}$ implies that either $x\in\mathfrak{p}$ or $y\in\mathfrak{p}.$ If each $\mathfrak{p}\in\text{Spec}(R)$ is $S$-strong then we say $S$ is a \emph{Strong extension} of $R.$
\end{definition}
\begin{proposition}
Suppose $G$ is locally finite. If $R\subset S$ is strong, then so is $R^G\subset S^G.$
\end{proposition}
\begin{proof}
Let $\mathfrak{p}\in\text{Spec}(R^G),$ then by LO property for the integral extension $R^G\subset R$ (Lemma \ref{l1}) there exists $P\in\text{Spec}(R)$ such that $\mathfrak{p}=P\cap R^G.$ Suppose $x,y\in S^G\subset S,$ $xy\in\mathfrak{p}.$ Then $xy\in P$ implies $x\in P$ or $y\in P,$ because $P$ is $S$-strong. Without loss of generality suppose $x\in P\subset R$, then it follows that $x\in R\cap S^G= R^G.$ Therefore $x\in P\cap R^G=\mathfrak{p}.$ Thus $R^G\subset S^G$ is Strong extension.
\end{proof}
If $R\subset S$ is a ring extension, then $R$ is said to be \emph{root-closed} in $S$ if, for any $x$ and $n\geq0,$ the containment $x^n\in R$ implies that $x\in R.$
\begin{proposition}
If $R\subset S$ root-closed, then so is $R^G\subset S^G.$
\end{proposition}
\begin{proof}
Take any $x\in S^G$ such that $x^n\in R^G,$ $n\geq0.$ Then $x^n\in R$ implies that $x\in R,$ by the root-closedness of $R.$ Therefore $x\in R\cap S^G= R^G.$
\end{proof}
\section{Properties of Pairs of Rings}\label{s3}
Let $\mathscr{P}$ be a ring theoretic property and $R\subseteq S$ be a ring extension. Recall that a pair $(R,S)$ is said to be a $\mathscr{P}$-pair, if $T$ satisfies $\mathscr{P}$ for each intermediate ring $R\subseteq T\subseteq S$ \cite{nz20}. Zeidi \cite{nz20} proved the invariance of several ring theoretic properties of pairs of rings under the operation $(R,S)\rightarrow (R^G, S^G).$ In particular, he proved that if $\mathscr{P}=\text{Residaully algebraic, LO, INC,
and Valuation}$ then, each of these properties pass from $(R,S)\rightarrow (R^G, S^G).$ The main purpose of this section is to contribute to this theme. We begin by proving the stability of GD-pair under the operation $(R,S)\rightarrow (R^G, S^G).$
\begin{definition}\label{3d1}
A pair of integral domains $(R,S)$ is said to be a \emph{Going-Down} pair (or \emph{GD}-pair), if the extension $R\subseteq T$ satisfies Going-Down hypothesis for each intermediate ring $R\subseteq T\subseteq S.$
\end{definition}
We now prove that the property of being a GD-pair passes from $(R,S)\rightarrow (R^G, S^G).$
\begin{theorem}\label{3t1}
Assume $(R,S)$ is a GD-pair and $G$ is finite. Then $(R^G, S^G)$ is also a GD-pair.
\end{theorem}
\begin{proof}
Let $(R,S)$ be a GD-pair and $R^G\subseteq T\subseteq S^G.$ Suppose $\mathfrak{p}_1\subseteq\mathfrak{p}_2$ be prime ideals in $R^G$ and there exists $\mathfrak{q}_2\in\text{Spec}(T)$ such that $\mathfrak{q}_2\cap R^G=\mathfrak{p}_2.$ We need to show that there exists $\mathfrak{q}_1\subseteq\mathfrak{q}_2$ satisfying $\mathfrak{q}_1\cap R^G=\mathfrak{p}_1.$

Since $R$ is integral over $R^G$ (Lemma \ref{l1}), therefore $T\subseteq RT$ is also integral extension. By Lying Over property of integral extensions, there exists $\mathfrak{p}'_{2}\in\text{Spec}(RT)$ such that $\mathfrak{p}'_2\cap T=\mathfrak{q}_2.$ Also note that $$\mathfrak{p}'_2\cap R^G=(\mathfrak{p}'_2\cap T)\cap R^G=\mathfrak{q}_2\cap R^G=\mathfrak{p}_2.$$
As $R^G\subseteq R\subseteq RT,$ we put $\mathfrak{p}'_2\cap R=\mathfrak{p}''_2.$ And $$\mathfrak{p}'_2\cap R\cap R^G=\mathfrak{p}''_2\cap R^G=\mathfrak{p}'_2\cap R^G=\mathfrak{p}_2.$$
By \cite[Proposition 2.11]{ds06} the ring extension $R^G\subseteq R$ satisfies GD, therefore there exists $\mathfrak{p}''_1\in\text{Spec}(R)$ satisfying $\mathfrak{p}''_1\subseteq \mathfrak{p}''_2$ and $\mathfrak{p}''_1\cap R^G=\mathfrak{p}_1.$ Since the extension $R\subseteq RT$ is GD, so there exists $\mathfrak{p}'_1\in\text{Spec}(RT)$ such that $\mathfrak{p}'_1\subseteq \mathfrak{p}'_2$ and $\mathfrak{p}'_1\cap R=\mathfrak{p}''_1.$ Put $\mathfrak{q}_1=\mathfrak{p}'_1\cap T.$ The ideal $q_1$ is the required prime ideal we are looking for. To see this, note that $$\mathfrak{q}_1\cap R^G=(\mathfrak{p}'_1\cap R^G\cap T)=\mathfrak{p}'_1\cap R^G=(\mathfrak{p}'_1\cap R)\cap R^G=\mathfrak{p}''_1\cap R^G=\mathfrak{p}_1.$$ It remains to show that $\mathfrak{q}_1\subseteq \mathfrak{q}_2,$ but this follows from the containment $\mathfrak{p}'_1\subseteq \mathfrak{p}'_2.$ This proves that $(R^G, S^G)$ is a GD-pair.
\end{proof}

Recall from \cite{ds06} that an integral domain $R$ is called \emph{Pseudo-valuation domain (PVD)} if there exists a (uniquely determined) valuation overring $V$ of $R$ such that $\text{Spec}(R)=\text{Spec}(V)$ (as sets). Before embarking on to the next result we recall the following characterization of PVD-pair from \cite[Lemma 5]{jt16}:
\begin{lemma}\label{l2}
Let $R\subset S$ be an extension of integral domains. Then the following hold true:
\begin{enumerate}
\item[(i)] If $R$ is a field, then $(R, S)$ is a PVD-pair if and only if $S$ is a field algebraic
over $R.$
\item[(ii)] If $R$ and $S$ are not fields, then $(R, S)$ is a PVD-pair if and only if $R$ is a
PVD with maximal ideal $\mathfrak{m}$ and associated valuation overring $V$ and either $S\subseteq V$ and $S/\mathfrak{m}$ is algebraic over $R/\mathfrak{m}$ or $S$ is an overring of $V$ and $V/\mathfrak{m}$ is algebraic over $R/\mathfrak{m}.$
\item[(iii)] If $R$ is not a field and $S$ is a field, then $(R, S)$ is a PVD-pair if and only if $R$
is a PVD with maximal ideal $\mathfrak{m}$ and associated valuation overring $V,$ $V/\mathfrak{m}$ is algebraic over $R/\mathfrak{m}$ and $S=\text{qf}(R)$ if and only if $R$ is a PVD with associated valuation overring $V$, $V = R'$ and $S = \text{qf}(R).$ Here $R'$ denotes the integral closure of $R.$
\end{enumerate}
\end{lemma}
We are now ready to prove the $G$-invariance of PVD-pair $(R,S).$
\begin{theorem}\label{3t2}
Assume $(R,S)$ is a PVD-pair and $G$ is finite. Then $(R^G, S^G)$ is also a PVD-pair.
\end{theorem}
\begin{proof}
We shall use Lemma \ref{l2}, therefore there are three separate cases. Suppose $(R,S)$ is a PVD-pair.
\begin{enumerate}
\item[\emph{Case 1}.] \emph{When $R$ is a field}\\
For this case one can adapt the first seven lines of the proof of \cite[Theorem 2.12]{nz20}. For the sake of completeness we provide the full proof here. By Lemma \ref{l1}, $R$ is integral over $R^G,$ therefore $R$ is a field if and only if $R^G$ is a field. Since $(R,S)$ is a PVD-pair, by Lemma \ref{l2}(i), $S$ is a field algebraic over $R.$ Once again invoking Lemma \ref{l1}, we see that $S^G$ is also a field. Moreover, $R^G\subset R\subset S,$ it follows from the transitivity of the algebraic property that $S$ is algebraic over $R^G.$ But $S^G\subset S,$ hence $R^G\subset S^G$ is an algebraic extension. Thus by Lemma \ref{l2} (i), $(R^G, S^G)$ is a PVD-pair.
\item[\emph{Case 2}.] \emph{When $R$ and $S$ are not fields}\\
It follows by the integrality of $R^G\subset R$ that $R^G$ is not a field. Similarly $S^G$ is also not a field. It follows Lemma \ref{l2}(ii), $R$ is a
PVD with maximal ideal $\mathfrak{m}$ and associated valuation overring $V.$ By \cite[Corollary 2.17]{ds06}, $(R^G,\mathfrak{m}^G)$ is a PVD with canonically associated domain $V^G.$ If $S\subseteq V,$ then $S^G\subseteq V^G.$ Since $R^G\subseteq R$ algebraic, hence, by \cite[Proposition 5.6]{am}, $R^G/\mathfrak{m}^G\subseteq R/\mathfrak{m}$ is also algebraic. But $R^G/\mathfrak{m}^G\subseteq R/\mathfrak{m}\subseteq S/\mathfrak{m},$ therefore $S/\mathfrak{m}$ is algebraic over $R^G/\mathfrak{m}^G.$ It follows that $S^G/\mathfrak{m}^G$ is algebraic over $R^G/\mathfrak{m}^G.$ If $S$ is an overring of $V,$ then $S^G$ is also an overring of $V^G.$ Using arguments similar to the above on can easily prove that $V^G/\mathfrak{m}^G$ is algebraic over $R^G/\mathfrak{m}^G.$ Thus once again by Lemma \ref{l2}(ii), it follows that $(R^G, S^G)$ is PVD-pair.
\item[\emph{Case 3}.] \emph{When $R$ is not a field but $S$ is a field}\\
Once again by invoking Lemma \ref{l1}, we can easily see that $R^G$ is nor a field and $S^G$ is a field. We have already seen that  $(R^G,\mathfrak{m}^G)$ is a PVD with canonically associated domain $V^G,$ and that $V^G/\mathfrak{m}^G$ algebraic over $R^G/\mathfrak{m}^G.$ It remains to be shown that $S^G=\text{qf}(R^G).$ By Lemma \ref{l2}(iii) we have$S=\text{qf}(R).$ But $\text{qf}(R^G)=\text{qf}(R)^G,$ by \cite[Lemma 2.3(b)]{ds06}. Hence $S^G=\text{qf}(R^G).$ Thus $(R^G, S^G)$ is PVD-pair (Lemma \ref{l2}(iii)).
\end{enumerate}
\end{proof}
One of the popular theme in the study of intermediate rings between $R$ and $S$ is the ``finite length
of intermediate chains of domains" property (for short FICP) between $R$ and $S$ \cite{mb}. We recall below from Nasr \cite{mb} the definition of FICP.
\begin{definition}
An extension of integral domains $R\subseteq S$ is said to have the \emph{finite length
of intermediate chains of domains} property (or \emph{FICP}) if each chain of intermediate rings between $R$ and $S$ is finite.
\end{definition}

\begin{theorem}\label{3t3}
Suppose $G$ is finite and $|G|\in U(R).$ If an $R\subseteq S$ has an FICP then so does $R^G\subset S^G.$
\end{theorem}
\begin{proof}
We shall prove that the map $T\rightarrow RT$ from $[R^G, S^G]$ to $[R,S]$ is injective. Suppose $T_1\neq T_2$ and $R^G\subseteq T_1\subset T_2\subseteq S.$ For if $RT_1=RT_2,$ then $t_2\in T_2-T_1$ belongs to $RT_1.$ Therefore $$t_2=\sum_{j=1}^{l}r_j t_j,$$ where $r_j\in R$ and $t_j\in T_1$ for each $j.$ Applying the $\sigma\in G$ to the both sides of the above equation, summing over $\sigma\in G$ and keeping in mind that $t_j\in S^G,$ we get $$|G|t_2=\sum_{j=1}^{l}\hat{r_j}t_j.$$ Since $|G|\in U(R),$ therefore $t_2=\dfrac{1}{|G|}\displaystyle\sum_{j=1}^{l}\hat{r_j}t_j \in T_1,$ a desired contradiction. Hence the map is injective and thus $[R^G, S^G]$ satisfies FICP.
\end{proof}

\end{document}